\documentclass[10pt,leqno]{amsart}
\usepackage{amssymb, amsmath}
\usepackage{psfrag,graphicx}
\usepackage[all]{xy}

\let\oldsection=\section

\makeatletter
\renewcommand{\@seccntformat}[1]{\bf\@nameuse{the#1}.\quad}
\renewcommand\section{\@startsection{section}{1}%
            \z@{.7\linespacing\@plus\linespacing}{.5\linespacing}%
            {\normalfont\bfseries \boldmath}}
\renewcommand\subsection{\@startsection{subsection}{2}%
            \z@{.5\linespacing\@plus.7\linespacing}{-.5em}%
            {\normalfont\bfseries \boldmath}}
\renewcommand\subsubsection{\@startsection{subsubsection}{3}%
            \z@{.3\linespacing\@plus.5\linespacing}{-.5em}%
            {\normalfont\bfseries \boldmath}}
\makeatother

\theoremstyle{plain}
\newtheorem*{thm}{Theorem}

\newtheorem*{prop}{Proposition}

\theoremstyle{definition}

\numberwithin{equation}{subsection}
\newcounter{listequation}

\def\note#1{{\small\tt <<#1>>}}  
\def\note#1{}             
\def\Z{\mathbb Z}

\def\:{\colon}

\def\la{\lambda}

\def\phi{\varphi}

\def\:{\colon}

\def\JWw0S{{}^J W^{-w_0 S}}
\def\w0JWS{{}^{-w_0 J} W^S}
\def\fg{{\mathfrak g}}

\def\fb{{\mathfrak b}}

\def\fn{{\mathfrak n}}

\def\Hom{\operatorname{Hom}}
\def\Ext{\operatorname{Ext}}
\def\H{\operatorname{H}}

\def\Der{\operatorname{Der}}
\def\Inn{\operatorname{Inn}}

\def\TJJ'{T_J^{J'}}
\def\TJ'J{T_{J'}^J}

\def\Zmu{Z^{+}(\mu)}

\def\Label#1{\label{#1}{\tt [#1]}}
\def\Label{\label}

\begin{document}
\title[$\operatorname{EXT}^{1}$-quivers for the Witt Algebra $W(1,1)$]
{$\operatorname{EXT}^{1}$-quivers for the Witt Algebra $W(1,1)$}

\author{Brian D. Boe}
\address{Department of Mathematics \\
                      University of Georgia\\ Athens, Georgia 30602  }
\email{brian@math.uga.edu, nakano@math.uga.edu}

\author{Daniel K. Nakano}
\thanks{Research of the second author partially supported by NSF grant
DMS-0654169}

\author{Emilie Wiesner}
\address{Department of Mathematics \\
            Ithaca College\\ Ithaca, New York 14850  }
\email{ewiesner@ithaca.edu}

\date{\today}



\dedicatory{}

\begin{abstract} Let ${\mathfrak g}$ be the finite dimensional Witt  
algebra $W(1,1)$
over an algebraically closed field of characteristic $p>3$. It is  
well known that
all simple $W(1,1)$-modules are finite dimensional. Each simple  
module admits
a character $\chi\in {\mathfrak g}^{*}$. Given ${\chi}\in {\mathfrak  
g}^{*}$ one
can form the (finite dimensional) reduced enveloping algebra $u 
({\mathfrak g},\chi)$.
The simple modules for $u({\mathfrak g},\chi)$ are precisely those  
simple $W(1,1)$-modules
admitting the character $\chi$. In this paper the authors compute $ 
\Ext^{1}$ between pairs of
simple modules for $u(\mathfrak g,\chi)$.
\end{abstract}

\maketitle

\parskip=2pt

\section{Introduction}
\Label{S:Intro}

\subsection{} Let ${\mathfrak g}$ be a restricted Lie algebra over
an algebraically closed field $k$ of characteristic $p>0$, and
$u({\mathfrak g})$ its restricted enveloping algebra.
Block and Wilson \cite{BW} have shown that all restricted
simple Lie algebras over $k$ are either classical or of Cartan type
provided that the characteristic of $k$ is larger than 7.
Each classical Lie algebra can be realized as the Lie algebra of a
reductive connected algebraic group $G$ with associated root system $ 
\Phi$.
When $p$ is greater than or equal to the Coxeter number $h$ of $\Phi$
there is a famous conjecture of Lusztig \cite[II 7.20]{Jan},  
involving Kazhdan-Lusztig polynomials
which predicts the dimensions of the simple modules for $u({\mathfrak  
g})$.
In recent times, a general scheme-theoretic setting
has been used to relate the representation theory of the algebraic group
$G$ with the representation theories of the infinitesimal Frobenius  
kernels $G_{r}$.
In the case when $r=1$, modules for $G_{1}$ are equivalent to modules  
for $u({\mathfrak g})$.
Cline, Parshall and Scott \cite{CPS2,CPS3} have shown that the  
Lusztig Conjecture can be equivalently
formulated in a graded category (i.e., $G_{1}T$-modules) using  
vanishing criteria for
$\text{Ext}^{1}$-groups between simple modules.

Lie algebras of Cartan type fall into four infinite classes of  
algebras: $W$, $S$, $H$ and $K$.
In most of the cases these Lie algebras are ${\mathbb Z}$-graded with 
${\mathfrak g}=\oplus_{i=t}^{s} {\mathfrak g}_{i}$ ($t=-1$ for
algebras of types $W$, $S$, $H$ and $t=-2$ for algebras of type $K$)  
where the Lie subalgebra
${\mathfrak g}_{0}$ is a classical Lie algebra. The simple modules
for $u({\mathfrak g})$ have been classified but their precise
description depends on knowing the simple modules for $u({\mathfrak g} 
_{0})$
\cite{Sh1,Sh2,Sh3,N1,Ho1,Ho2,Ho3,Hu1,Hu2} (i.e., for $p$ large, this
reduces to knowing the Lusztig conjecture). Moreover, composition  
factors
of the projective indecomposable modules of $u({\mathfrak g})$ depend on
understanding the composition factors of the projective modules for
$u({\mathfrak g}_{0})$ (cf.\ \cite[Thm. 3.1.5]{N1}, \cite[Thm. 4.1] 
{HoN}). Holmes and Nakano
\cite{N1,HoN} also proved that the restricted enveloping algebra
$u({\mathfrak g})$ for a Lie algebra of Cartan type has one block.

\subsection{} From the above discussion, it is natural to ask about  
extensions
between simple modules for $u({\mathfrak g})$ where ${\mathfrak g}$ is
a Lie algebra of Cartan type. Since $u({\mathfrak g})$ has one block,
one would expect that the computation of extensions would be at least
as difficult as in the classical case. Lin and Nakano \cite{LN}  
proved that
for the Lie algebras of type $W$ and $S$ there are no self-extensions
between simple modules. This result parallels Andersen's \cite{A}  
earlier work
where he showed that self-extensions do not exist for simple $G_{1}$- 
modules unless
$p=2$ and $\Phi$ has a component of type $C_{n}$.

The main results in this paper aim to initiate the study of  
extensions between
simple modules for Cartan type Lie algebras, not just for restricted  
simple modules
but also for the non-restricted ones. Fix ${\mathfrak g}=W(1,1)$.  
Given $\chi\in {\mathfrak g}^{*}$,
one can form the reduced enveloping algebra $u({\mathfrak g},\chi)$.  
The simple
modules for these algebras were determined by Chang \cite{Ch}. A  
complete description of
the simple modules with their projective covers is also given in \cite 
[\S\S 1,2]{FN}. The
goal of this paper is to calculate $\text{Ext}^{1}_{u({\mathfrak g}, 
\chi)}(S,T)$
where $S,T$ are simple $u({\mathfrak g},\chi)$-modules. In Section 2,  
we begin investigating
the restricted case (i.e., when $\chi=0$). A complete description of $ 
\text{Ext}^{1}$ groups between
the $p$-dimensional induced modules $Z^{+}(\lambda)$ is given. With  
this information, we
compute the $\text{Ext}^{1}$-quiver for $u(W(1,1))$ in Section 3.  
Section 4 is devoted to
handling the non-restricted case. The height $r(\chi)$ 
of the character  is a useful invariant
to delineate the representation theory of $W(1,1)$. The restricted  
case coincides with the height
$r(\chi)=-1$. The extension information for the $r(\chi)=0$ and $r 
(\chi)=1$ case uses ideas from the
computations in the restricted situation. It is interesting to note  
that when $1<r(\chi)<p-1$
there is only one simple $u({\mathfrak g},\chi)$-module. The  
dimensions of the simple module
and its projective cover are known. However, we are only able to  
provide partial results on $\text{Ext}^{1}$ in this case. It appears that this calculation when $1<r(\chi) 
<p-1$ needs
a more detailed understanding of the representation theory which is  
not yet available.
For $r(\chi)=p-1$, we can use results from \cite[Theorem 2.6]{FN} to  
handle this case completely.

\subsection{Notation and Preliminaries} \Label{S:Preliminaries}
Throughout this paper let $k$ be an algebraically closed field of  
characteristic $p\geq 5$.
The $p$-dimensional Witt algebra ${\mathfrak g}:=W(1,1)$ is the $ 
{\mathbb Z}$-graded
restricted simple Lie algebra with basis $\{e_i \mid -1\leq i\leq p-2 
\}$;
with the bracket operation
$$
[e_i,e_j]=(j-i)e_{i+j}\qquad\qquad -1\leq i,j\leq p-2,
$$
where $e_{i+j}:=0$ if $i+j\not\in\{-1,...,p-2\}$;  and with the $p$- 
mapping given by
$e_i^{[p]}:=\delta_{i0} e_i$ for any $-1\leq i\leq p-2$.
We remark that one could also define this Lie algebra for $p=3$, but
then it would be isomorphic to $\mathfrak{sl}(2)$.

Let $\chi\in {\mathfrak g}^{*}$. The {\em centralizer\/} of $\chi$ in  
${\mathfrak g}$ is
defined as
$$
{\mathfrak g}^{\chi}=\{\, x\in {\mathfrak g} \mid \ \chi([x, 
{\mathfrak g}])=0\,\}.
$$
Consider the $p$ subalgebras
$$
{\mathfrak g}_i:=ke_i\oplus\cdots\oplus ke_{p-2}\qquad\qquad -1\leq i
\leq p-2.
$$
We set ${\mathfrak b}^{+}={\mathfrak g}_{0}$,
${\mathfrak t}=ke_{0}$, and ${\mathfrak n}^{+}={\mathfrak g}_{1}$.
Define the
{\em height\/} $r(\chi)$ of $\chi$ by
$$
r(\chi):=\min\{\, i \mid -1\leq i\leq p-2 \text{ and } \chi|_ 
{{\mathfrak g}_i}=0\,\}
$$
if $\chi(e_{p-2})=0$ and $r(\chi):=p-1$ if $\chi(e_{p-2})\neq 0$.
The height of $\chi$ turns out to be good invariant for studying the
simple modules for ${\mathfrak g}$.

If $\chi\in {\mathfrak g}^{*}$, the {\em reduced enveloping algebra}
$u({\mathfrak g},\chi)$ is a finite dimensional self-injective algebra.
The algebra $u({\mathfrak g},\chi)$ is defined as the quotient of the  
universal
enveloping algebra $U({\mathfrak g})$ by the ideal generated by  
(central)
elements of the form $x^{p}-x^{[p]}-\chi(x)^{p}$ for $x\in {\mathfrak  
g}$.
Any simple module for ${\mathfrak g}$ is finite dimensional and  
admits a character
$\chi\in {\mathfrak g}^{*}$ such that $x^{p}-x^{[p]}$ acts as $\chi(x) 
^{p}$. In
this way every simple module for ${\mathfrak g}$ is a simple module  
for $u({\mathfrak g},\chi)$
for some $\chi$. If $\chi=0$ then modules for $u({\mathfrak g}):=u 
({\mathfrak g},0)$ correspond to
restricted representations for ${\mathfrak g}$.

\section{Extensions Between Verma Modules}

\subsection{} For the remainder of Section 2 and Section 3, we will  
assume that $\chi=0$
or equivalently $r(\chi)=-1$.

The simple modules for $u({\mathfrak t})$ are parametrized by  
elements in $X:={\mathbb F}_{p}$.
If $\lambda\in X$, we can let ${\mathfrak n}^{+}$ act trivially and  
obtain a one-dimensional
$u({\mathfrak b}^{+})$-module. Let $Z^{+}(\lambda)=u({\mathfrak g}) 
\otimes_{u({\mathfrak b}^{+})}\lambda$, the Verma module of highest  
weight $\lambda$.
Then $Z^+(\la)$ is $p$-dimensional and has a basis $\{m_0, \ldots, m_ 
{p-1}\}$ where the action
of ${\mathfrak g}$ is given by
\begin{equation}\Label{moduleaction}
e_k. m_j = \left( j+k+1 + (k+1) \la \right) m_{j+k}
\end{equation}
(cf.\ \cite[Lemma 2.2.1]{N1}). Note that $m_{p-1}$ is a highest  
weight vector of weight $\lambda$ for $\mathfrak t$.

In order to compute cohomology it will be important to compute spaces  
of derivations. For
our first computation we define
\begin{eqnarray*}
\Der (\fn^+, Z^+(\la)) &=& \{ d: \fn^+ \rightarrow Z^+(\la)) \mid d 
([x,y])=x.d(y)-y.d(x) \}, \\
\Inn (\fn^+, Z^+(\la)) &=& \{ d: \fn^+ \rightarrow Z^+(\la)) \mid  
\exists m \in Z^+(\la) \ \mbox{such that} \ d(x)=x.m \},
\end{eqnarray*}
the set of $Z^{+}(\lambda)$-valued derivations (respectively, inner  
derivations) of $\fn^{+}$. Note that $\Inn (\fn^+, Z^+(\la))  
\subseteq \Der (\fn^+, Z^+(\la))$.  The space $ \Der (\fn^+, Z^+(\la)) 
$ is a $\fg$-module with the $\fg$-action given by $(x.d)(m)=x.(d(m))- 
d(x.m)$ for $x \in \fg$ and $m \in Z^+(\la)$.

\subsection{}\Label{SS:Vermaextension} The following theorem gives a  
description of
$\text{Ext}^{1}_{u({\mathfrak g})}(Z^+(\mu),Z^{+}(\la))$ for all $p 
\geq 5$.

\begin{thm} \Label{Thm:ZExt} Let ${\mathfrak g}=W(1,1)$.
\begin{itemize}
\item[(a)] For $p=5$,
$$
\Ext^1_{u(\fg)}(Z^+(\mu), Z^+(\la)) \cong \left\{ \begin{array}{ll}
k & \mbox{if} \ \la - \mu=2,3 \\
& \mbox{or} \ \{\mu, \la\} =\{0 \}, \{ 4 \} , \{ 0,1 \}, \{ 0,4 \},  
\{3, 4\}
,  \\
0 & \mbox{otherwise}.
\end{array} \right.
$$
\item[(b)] For $p \geq 7$,
$$\Ext^1_{u(\fg)}(Z^+(\mu), Z^+(\la))\cong k$$
if
\begin{itemize}
\item[(i)] $\la - \mu=2,3,4$;
\item[(ii)] $(\mu, \la)=(0,0), (-1,-1), (-1,0), (0,-1), (0,1),  
(-2,-1), (-5,0), (-1,4)$;
\item[(iii)] $\lambda$ is such that $2 \la^2-10 \la +3=0$ and $\la- 
\mu=6$.
\end{itemize}
Otherwise, $\Ext^1_{u(\fg)}(Z^+(\mu), Z^+(\la))= 0$.
\end{itemize}
\end{thm}

The $\text{Ext}^{1}$-quivers between Verma modules
for $p=5,7$ are shown in Figure \ref{Fig:Ext1verma}.

\psfrag{0}{{\footnotesize$0$}}
\psfrag{1}{{\footnotesize$1$}}
\psfrag{2}{{\footnotesize$2$}}
\psfrag{3}{{\footnotesize$3$}}
\psfrag{4}{{\footnotesize$4$}}
\psfrag{5}{{\footnotesize$5$}}
\psfrag{6}{{\footnotesize$6$}}
\psfrag{7}{{\footnotesize$7$}}
\psfrag{8}{{\footnotesize$8$}}
\psfrag{9}{{\footnotesize$9$}}
\psfrag{10}{{\footnotesize$10$}}
\psfrag{p=5}{{$p=5$: ``Maresias''}}
\psfrag{p=7}{{$p=7$}}
\begin{figure}[h]
\centering
\includegraphics[height=.25\textheight]{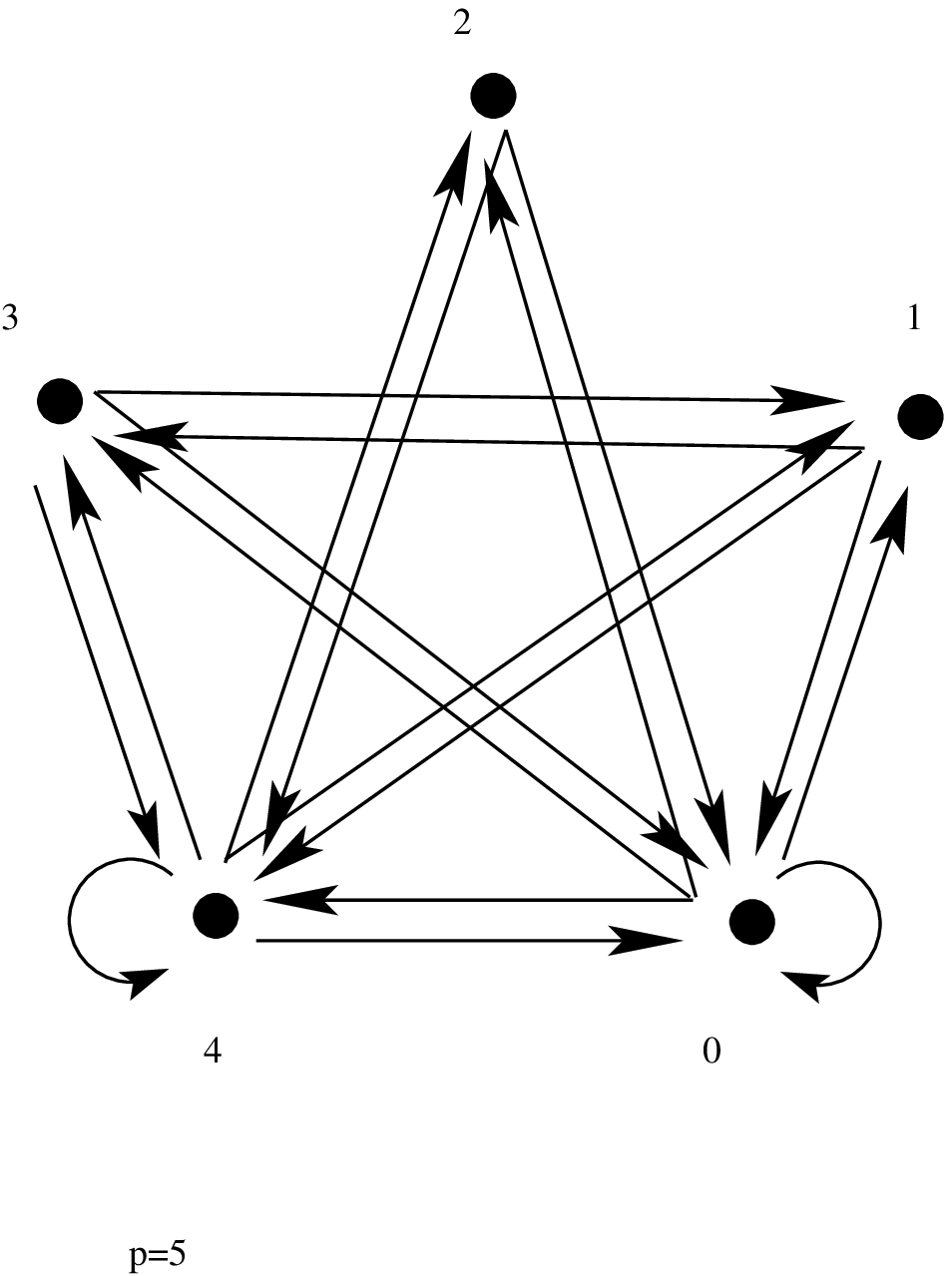}
\hspace{.5in}
\includegraphics[height=.3\textheight]{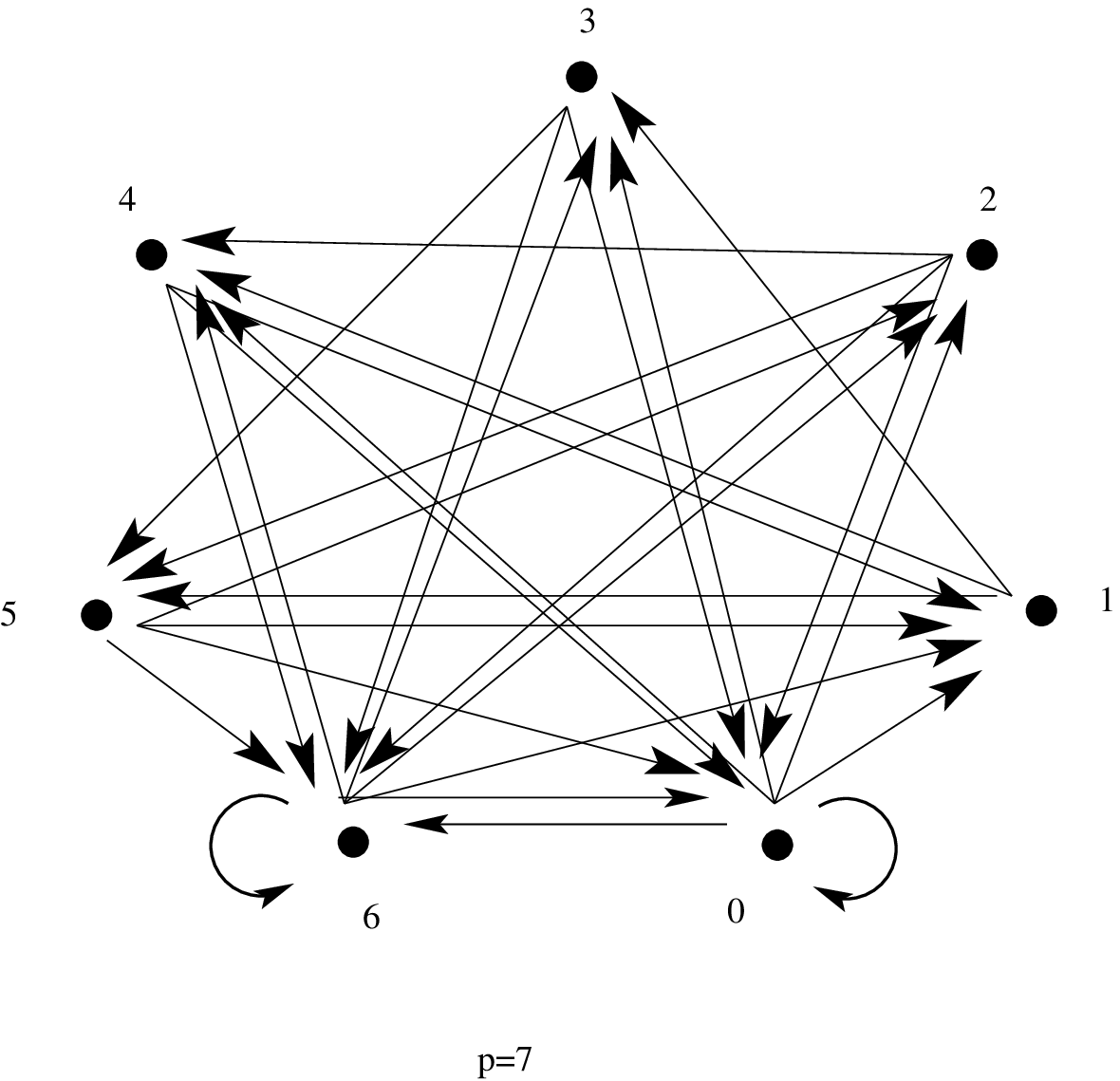}
\caption{$\Ext^1$-quivers for $Z^+(\la)$ ($p=5,7$)} \Label 
{Fig:Ext1verma}
\end{figure}

\subsection{}\Label{SS:VermaDercomp} We will prove the theorem in two  
main steps. Our first step is to
compute the weight spaces $\Der(\fn^+, Z^+(\la))_{\mu}$ and $\Inn(\fn^ 
+, Z^+(\la))_{\mu}$. There exists the following isomorphism relating   
the (ordinary) Lie algebra cohomology with derivations modulo inner  
derivations \cite[VII. Proposition 2.2]{HS}:
$$\text{H}^{1}(\fn^+,Z^+(\la))_{\mu}\cong \Der(\fn^+, Z^+(\la))_ 
{\mu}/ \Inn(\fn^+, Z^+(\la))_{\mu}.$$
The second step (in Section \ref{SS:ExtvsH1}) will be to equate $\Ext^ 
{1}$ between Verma modules with the weight spaces of the first Lie  
algebra cohomology. Surprisingly, the verification of the second step  
will rely on the explicit computation of derivations in the first step.

We first consider $\Inn(\fn^+, Z^+(\la))_{\mu}$.   If $0<j<p$, then $ 
\Inn(\fn^+, Z^+(\la))_{\la+j}$ is generated by the derivation $d_j$  
defined by
$$d_j(e_k)=e_k. m_{j-1}=
\begin{cases} \  (j+k+(k+1)\la)m_{j+k-1}\  &\text{ if } k \leq p-j \\  
\ 0\  &\text{ otherwise.}
\end{cases}
$$
Then $d_j =0$ if and only if $j=1$ and $\la=p-1$, or $j=p-1$ and $ 
\la=0$.  If $j=0$, there is no inner derivation. Therefore,
\begin{equation} \Label{Eq:Inn}
\Inn(\fn^+, Z^+(\la))_{\la+j} \cong
\begin{cases}
\ 0 \ \text{ if } j=0;\ j=1, \la=p-1; \text{ or }  j=p-1,\la=0, \\
\ k \ \text{ otherwise.}
\end{cases}
\end{equation}

To compute $\Der(\fn^+, Z^+(\la))_{\la+j}$, we first assume that $p>7 
$.  For $n\in\Z$, let $\overline n$ denote the unique integer  
satisfying $0\le \overline n < p$ and $\overline n \equiv n \pmod p$.
Given $d \in \Der(\fn^+, Z^+(\la))_{\la+j}$, we can write $d(e_k) =  
a_k m_{\overline{j+k-1}}$ for some $a_{k}\in k$.  Since $\fn^+$ is  
generated by $e_1$ and $e_2$, all $a_k$ are determined by $a_1$ and  
$a_2$ using the definition of a derivation, and thus $\dim \Der(\fn^+, Z^+ 
(\la))_{\la+j} \leq 2$ .
Observe that $d$ will be a derivation for any choice of $a_1, a_2 \in  
k$ such that
\begin{gather}
\Label{Eq:der1} d(e_k) \text{ is well-defined for each } 2<k\leq p-2; \\
\Label{Eq:der2} d([e_i,e_k])=0 \text{ for } i+k>p-2.
\end{gather}

Suppose $p-4 \leq j \leq p-2$. Since $e_3$ and $e_4$ are uniquely  
generated by $e_1$ and $e_2$, $d(e_3)$ and $d(e_4)$ are well-defined  
for any choice of $a_1$ and $a_2$.   Moreover, the action of $\fn^+$  
respects the grading on $Z^+(\la)$.  This implies that $d(e_k)=0$ for  
$k>p-j$, so that \eqref{Eq:der1} holds, and \eqref{Eq:der2} holds.   
(For example, if $j=p-3$, $d(e_3)= a_2 e_1 . m_{p-1} + a_1 e_2 . m_ 
{p-2}=0$.) Therefore, for $p-4 \leq j \leq p-2$, $\dim \Der(\fn^+, Z^+ 
(\la))_{\la+j} = 2$.  Since $\dim \Inn(\fn^+, Z^+(\la))_{\la+j}=1$ in  
this case, we have
$$\text{H}^{1}(\fn^+,Z^+(\la))_{\mu}\cong k \quad \mbox{if} \  \la -  
\mu=2,3,4.$$

Suppose $0 \leq j \leq p-5$.  Then for $k>4$, $d(e_k)$ is not  
necessarily well-defined.  For example,
\begin{eqnarray}
d(e_5) &=& d([e_2,e_3])=d([e_2, [e_1, e_2]])  \Label{Eq:e51} \\
&=& \left[ \left((j+3+2\la)(j+5+3\la)-(j+5+4\la)\right) a_2 \right.  
\nonumber \\
&&\left. -(j+3+3\la)(j+5+3\la)a_1\right] m_{j+4}  \nonumber \\
d(e_5) &=&  \frac{1}{3}d([e_1,e_4])=\frac{1}{6} d([e_1,[e_1, 
[e_1,e_2]]]) \Label{Eq:e52} \\
&=& \frac{1}{6} \left[ (j+3+2\la)(j+4+2\la)(j+5+2\la)a_2 \right.  
\nonumber \\
&& - \left((j+3+3\la)(j+4+2\la)(j+5+2\la)+(j+4+4\la)(j+5+2\la)  
\right. \nonumber \\
&& \left. \left. +2(j+5+5\la) \right)a_1 \right] m_{j+4}. \nonumber
\end{eqnarray}
If $d(e_5)$ is well-defined, then
\begin{eqnarray}
0&=& \mbox{\eqref{Eq:e51}} - \mbox{\eqref{Eq:e52}}  = \left( p_{1,5} 
(j,\la)a_1 + p_{2,5}(j,\la)a_2 \right) m_{j+4}, \Label{Eq:e53}
\end{eqnarray}
where
\begin{eqnarray}
p_{1,5}(j,\la)&=&\frac{1}{6} \left(10j+8\la+7j^2+27j\la+20\la^2+16j 
\la^2+j^3+12\la^3+7j^2\la \right) \\
p_{2,5}(j,\la)&=&- \frac{1}{6} \left(6j^2+5j+18j\la+4\la+12\la^2+6j^2 
\la+j^3+12j\la^2+8\la^3 \right).
\end{eqnarray}
Any derivation respects the Jacobi identity, so that $d([e_1,e_5])= d 
([e_1, [e_2, e_3]])=2d([e_2, e_4])$.  Therefore, if $d(e_5)$ is 
well-defined, then $d(e_6)$ is also well-defined.  For $d(e_7)$ to be 
well-defined we must have
\begin{eqnarray}
0&=&d([e_3,e_4])-\frac{1}{3}d([e_2,e_5]) = \left( p_{1,7}(j,\la)a_1+p_ 
{2,7}(j,\la)a_2 \right)m_{j+6},  \Label{Eq:e7}
\end{eqnarray}
where
\begin{eqnarray*}
p_{1,7}(j,\la)&=& \frac{1}{6}\left(14j+12\la+9j^2+39j\la+30\la^2+24j 
\la^2+j^3+18\la^3+9j^2\la \right),\\
p_{2,7}(j,\la)&=&-\frac{1}{6}\left(7j+6\la+8j^2+26j\la+18\la^2+18j 
\la^2+j^3+12\la^3+8j^2\la\right).
\end{eqnarray*}
Note that if $j=p-5$, $p-6$ then in fact $d(e_7)=0$ using the grading  
on $Z^+(\la)$.  In general, it will be enough to consider $d(e_5)$  
and $d(e_7)$ to determine when \eqref{Eq:der1} holds.

We use the Jacobi identity to reduce \eqref{Eq:der2} to the cases $d 
([e_1, e_{p-2}])=0$ and $d([e_2, e_k])=0$ for $k=p-2$, $p-3$.  Because  
the action of $\fn^+$ respects the grading on $Z^+(\lambda)$, these  
equations automatically hold for $2 \leq j \leq p-5$.
For $j=0$, the remaining equations are
\begin{eqnarray}
0&=& d([e_2,e_{p-2}])  = \frac{1}{p-4} d([e_2, [e_1, e_{p-3}]])   
\Label{Eq:j=0a} \\
\nonumber &=& \frac{2}{p-4}\la(3(\la-1)a_{p-3}+3(\la+1)a_1-2a_2)m_ 
{p-1};\\
0&=& d([e_1, e_{p-2}]) = \frac{1}{p-4}d([e_1,[e_1,e_{p-3}]]) \Label 
{Eq:j=0b} \\
\nonumber &=& \frac{2}{p-4}((\la-1)(2\la-1)a_{p-3}+(\la+1)(2\la-3)a_1) m_{p-2};\\
0&=&d([e_2, e_{p-3}])=\left((3\la-1)a_{p-3}+(2\la+1)a_2\right)m_ 
{p-2}. \Label{Eq:j=0c}
\end{eqnarray}
For $j=1$,
\begin{eqnarray}
0&=& d([e_1, e_{p-2}])
=\frac{2}{p-4}\la(2\la-1)(a_{p-3}+a_1)m_{p-1}; \Label{Eq:j=1a} \\
0&=&d([e_2,e_{p-3}])= \la(3a_{p-3}+2a_2)m_{p-1}.  \Label{Eq:j=1b}
\end{eqnarray}
If $j=1$, $d([e_2, e_{p-2}])=0$ automatically from the grading.

Let $j= p-5$.  From \eqref{Eq:e53}, we see that  $\dim \Der(\fn^+, Z^+ 
(\la))_{\la-5}=2$ if
\begin{equation} \Label{Eq:ZExt5}
0=p_{1,5}(p-5,\la)=2\la(\la-1)(\la-4)\quad\text{and} \quad 0= p_{2,5} 
(p-5,\la)=-\frac{4}{3}\la(\la-2)(\la-4).
\end{equation}
Otherwise $\dim \Der(\fn^+, Z^+(\la))_{\la-5} \leq 1$.  Since $\dim  
\Inn(\fn^+, Z^+(\la))_{\la-5}=1$,
$$
\text{H}^{1}(\fn^+,Z^+(\la))_{\la-5}\cong \left\{ \begin{array}{ll}
k & \mbox{if} \ \la=0,4, \\
0 & \mbox{otherwise.}
\end{array} \right.
$$

If $j=p-6$, we may apply the same arguments.  In this case, $0=p_{1,5} 
(p-6,\la)=p_{2,5}(p-6,\la)$ if and only if $\la$ is such that
$2\la^2-10\la+3=0$. Thus
$$
\text{H}^{1}(\fn^+,Z^+(\la))_{\la-6}\cong \left\{ \begin{array}{ll}
k & \mbox{if } \ 2\la^{2}-10\la+3=0, \\
0 & \mbox{otherwise.}
\end{array} \right. 
$$

For $0 \leq j \leq p-7$, $\dim \Der(\fn^+, Z^+(\la))_{\la+j}=2$ only if
\begin{equation} \Label{Eq:Zmiddlecase}
0=p_{1,5}(j,\la)=p_{2,5}(j,\la)=p_{1,7}(j,\la)=p_{2,7}(j,\la).
\end{equation}
Observe that
\begin{equation}  \Label{Eq:Zmiddlecase2}
2\left( p_{1,7}(j,\la)+2p_{2,7}(j,\la) \right)-3\left( p_{1,5}(j,\la) 
+2p_{2,5}(j,\la) \right)=j(j+1)(j+\la).
\end{equation}
Therefore, \eqref{Eq:Zmiddlecase} does not hold for $j \neq 0,p-1,-\la 
$.  By checking $j=0,p-1,-\la$ directly, we find that \eqref 
{Eq:Zmiddlecase} holds precisely for $(j, \la)=(1,p-1), (0,0),(0,p-1) 
$.  Thus, if $1 \leq j \leq p-7$ and $(j, \la) \neq (1, p-1)$, then $ 
\dim \Der (\fn^+, Z^+(\la))_{\la+j}=1$, which implies
$\text{H}^{1}(\fn^+,Z^+(\la))_{\la+j}=0$.

For $(j, \la)=(1, p-1)$, there is a derivation $d \in \Der(\fn^+, Z^+ 
(p-1))_{0}$ defined by $d(e_k)=(k+1)m_k$.  Since \eqref{Eq:j=1a} and  
\eqref{Eq:j=1b} are not identically zero for $\la=p-1$, there is not  
another linearly independent solution.  Using $\Inn(\fn^+, Z^+(p-1))_ 
{0}=0$, we have that $\text{H}^{1}(\fn^+,Z^+(p-1))_{0}\cong k$.

If $j=0$, \eqref{Eq:e53} and \eqref{Eq:e7} are linearly dependent.   
It is straightforward to check that \eqref{Eq:e53}, \eqref{Eq:j=0a},  
\eqref{Eq:j=0b}, and \eqref{Eq:j=0c} have a common solution only for $ 
\la=0$ (corresponding to the derivation $d(e_k)=\frac{k(k+1)}{2}m_ 
{k-1}$) and for $\la=p-1$ (corresponding to the derivation $d(e_1)=m_0 
$, $d(e_k)=0$ for $k\neq1$).  Therefore,
$$
\text{H}^{1}(\fn^+,Z^+(\la))_{\la}\cong \left\{ \begin{array}{ll}
k & \mbox{if} \ \la=0,\, p-1, \\
0 & \mbox{otherwise.}
\end{array} \right. 
$$

Suppose  $j=p-1$.  Then $d: \fn^+ \rightarrow Z^+(\la)$ given by $d 
(e_1)=m_{p-1}$, $d(e_k)=0$ for $k>1$ is a well-defined derivation,  
and so $\dim  \Der(\fn^+, Z^+(\la))_{\la-1}  \geq 1$. Now let $d \in  
\Der (\fn^+, Z^+(\la))_{\la-1}$ be arbitrary.  Since $d(e_1)=a_1 m_ 
{p-1}$, the grading on $Z^+(\la)$ implies that
\begin{eqnarray}
d(e_5) &=& d([e_2, [e_1, e_2]])  \Label{Eq:e51a} \\
&=& \left((j+3+2\la)(j+5+3\la)-(j+5+4\la)\right) a_2 m_{3}  \nonumber \\
d(e_5) &=& \frac{1}{6} d([e_1,[e_1[e_1,e_2]]]) \Label{Eq:e52a} \\
&=& \frac{1}{6} \left[ (j+3+2\la)(j+4+2\la)(j+5+2\la)a_2  \right]m_ 
{3} \nonumber.
\end{eqnarray}
Note that these equations correspond to the terms containing $a_2$  
in  \eqref{Eq:e51} and \eqref{Eq:e52}.  Thus, $d(e_5)$ is well- 
defined if and only if $0=p_{2,5}(-1, \la)a_2=-\frac{4}{3}\la(\la-1) 
(\la+1)a_2$.  Similarly, $d(e_7)$ is well-defined if and only if $0=p_ 
{2,7}(-1,\la)a_{2}=-2\la(\la-1)(\la+1)a_2$. Thus, for $a_2 \neq 0$, $d 
(e_k)$ is well-defined for all $k$ only if $\la =0,1,-1$.  Recalling  
\eqref{Eq:der2}, we observe
\begin{eqnarray}
0&=& d([e_2,e_{p-2}]) = \frac{1}{p-4}\left((3\la-1)(2\la-3)a_{p-3} -4 
(\la+1)a_2 \right)m_{p-2};  \Label{Eq:j=-1a}  \\
0&=& d([e_1, e_{p-2}])  =  \frac{2}{p-4}(\la -1)(2\la-3) a_{p-3}m_ 
{p-3}. \Label{Eq:j=-1b}
\end{eqnarray}
When $\la=0$, these equations have no common solution with $a_2 \neq 0 
$ so that $\dim  \Der(\fn^+, Z^+(0))_{-1}  =1$.   It can be checked  
that for $\la =1$, $d(e_1)=0$ and $d(e_k)={k+1 \choose 3}  m_{k-2}$,  
$k \neq 1$, is a well-defined derivation; and for $\la =-1$, $d(e_2) 
=m_0$, $d(e_k) =0$, $k \neq 1$, is a well-defined derivation.  This  
implies $\dim  \Der(\fn^+, Z^+(\la))_{\la-1}  =2$ if $\la =1,-1$.   
Recalling that $\Inn (\fn^+, Z^+(0))_{-1}=0$, we have
$$
\text{H}^{1}(\fn^+,Z^+(\la))_{\la-1} \cong
\left\{
\begin{array}{ll}
k & \mbox{if} \ \la=0,1,-1, \\
0 & \mbox{otherwise.}
\end{array} \right. 
$$

Now suppose $p=7$.  Then $\Der(\fn^+, Z^+(\la))_{\la+j}$ can be  
computed as above for $2 \leq j \leq 6$ and for $j=0$.     For  
$j=1=7-6$, the reasoning above  shows that $\dim \Der(\fn^+, Z^+(\la)) 
_{\la+j}\leq 1$ for all $\la$, since $2\la^2-10\la+3 \neq 0$, and $ 
\Der(\fn^+, Z^+(p-1))_{0}\cong k$.   This produces the stated results.

Finally, assume $p=5$. Then $\Der(\fn^+, Z^+(\la))_{\la+j}$ can be  
computed as above for $2=5-3 \leq j \leq 3=5-2$.  For $j=1$ and $d  
\in \Der(\fn^+, Z^+(\la))_{\la+1}$ it is only necessary to check that  
\eqref{Eq:j=1a} and \eqref{Eq:j=1b} hold.  Note that \eqref{Eq:j=1b}  
always holds, and \eqref{Eq:j=1a} holds if $a_1=-a_2$ or $\la=0,3$.  
Combining this with \eqref{Eq:Inn}, we have
$$
\text{H}^{1}(\fn^+,Z^+(\la))_{\la+1}
\cong \left\{ \begin{array}{ll}
k & \mbox{if} \ \la=0,3, \\
0 & \mbox{otherwise.}
\end{array} \right. 
$$
For $j=0$,  $d \in \Der(\fn^+, Z^+(\la))_{\la}$ if and only if \eqref 
{Eq:j=0a}, \eqref{Eq:j=0b}, and \eqref{Eq:j=0c} hold.  There is one  
solution for these equations for $\la=0,4$ and no solutions for other  
$\la$.  Since $\Inn(\fn^+, Z^+(\la))_{\la}=0$,
$$
\text{H}^{1}(\fn^+,Z^+(\la))_{\la}\cong \left\{ \begin{array}{ll}
k & \mbox{if} \ \la=0,4, \\
0 & \mbox{otherwise.}
\end{array} \right. 
$$

In the case $j=p-1$, $\dim \text{Der}(\fn^+, Z^+(\la))_{\la-1}=2$ if  
and only if \eqref{Eq:j=-1a} and \eqref{Eq:j=-1b} hold.  This occurs  
for $\la=1,4$.  For $\la =0$, we have $\dim  \text{Der}(\fn^+, Z^+ 
(\la))_{\la-1}=1$ since $d(e_1)=m_{p-1}$ is a well-defined  
derivation. Therefore,
$$
\text{H}^{1}(\fn^+,Z^+(\la))_{\la-1}\cong \left\{ \begin{array}{ll}
k & \mbox{if} \ \la=0,1,4, \\
0 & \mbox{otherwise.}
\end{array} \right. 
$$

\subsection{} \Label{SS:ExtvsH1} In the following proposition, we  
identify the restricted Lie algebra cohomology  for extensions  
between Verma modules
with the ordinary Lie algebra cohomology.

\begin{prop} \Label{P:Derivation} Let $\lambda, \mu\in X$. Then
$$\Ext^1_{u(\fg)} (Z^+(\mu), Z^+(\la)) \cong \operatorname{H}^{1}(\fn^ 
+,Z^+(\la))_{\mu}\cong
\left(\Der(\fn^+, Z^+(\la))/\Inn(\fn^+, Z^+(\la))\right)_{\mu} .$$
\end{prop}

\begin{proof} The second isomorphism was stated earlier so it  
suffices to verify
the first one. First one can apply Frobenius reciprocity to obtain
$$\Ext^1_{u(\fg)}(Z^+(\mu), Z^+(\la))\cong \text{Ext}^{1}_{u 
({\mathfrak b}^{+})}(\mu, Z^{+}(\la)).$$
Now ${\mathfrak n}^{+}$ is an ideal in ${\mathfrak b}^{+}$ with $ 
{\mathfrak b}^{+}/{\mathfrak n}^{+}\cong
{\mathfrak t}$. By applying the Lyndon-Hochschild-Serre spectral  
sequence and the fact that
modules over $u({\mathfrak t})$ are semisimple, one sees that
\begin{eqnarray*}
\text{Ext}^{1}_{u({\mathfrak b}^{+})}(\mu, Z^{+}(\la))&\cong&
\text{Ext}^{1}_{u({\mathfrak n}^{+})}(\mu, Z^{+}(\la))^{u(\mathfrak  
t)}\\
&\cong& \text{Ext}^{1}_{u({\mathfrak n}^{+})}(k, Z^{+}(\la))_{\mu}\\
&\cong& \text{H}^{1}(u({\mathfrak n}^{+}),Z^{+}(\la))_{\mu}.
\end{eqnarray*}

According to \cite[I 9.19]{Jan}, there exists an injection
$$0\rightarrow \text{H}^{1}(u({\mathfrak n}^{+}),Z^{+}(\la))_{\mu} 
\hookrightarrow
\text{H}^{1}({\mathfrak n}^{+},Z^{+}(\la))_{\mu}.$$
Moreover, $\phi\in \text{H}^{1}({\mathfrak n}^{+},Z^{+}(\la))_{\mu}$  
will be in
the image of this map if and only if $\phi(x^{[p]})=x^{p-1}.\phi(x)$  
for all $x\in {\mathfrak n}^{+}$.
However, $x^{[p]}=0$ for all $x\in {\mathfrak n}^{+}$ and the module  
$Z^{+}(\lambda)$ is
${\mathbb Z}$-graded so the aforementioned condition reduces to $e_{1} 
^{p-1}.\phi(e_{1})=0$.
Without loss of generality we may assume that $\phi(e_{1})=m_{0}$.  
 From a direct computation
with the action, $e_{1}^{p-1}.m_{0}= 0$ unless $\lambda=\frac{p-1}{2}$.

Now in this case $e_{0}.m_{0}=\frac{p+1}{2}m_{0}$, so $\phi$ must be  
a derivation
of weight $-1+\frac{p+1}{2}=\frac{p-1}{2}=\la$. However, the  
computations in the previous
section demonstrate that $\text{H}^{1}({\mathfrak n}^{+},Z^{+} 
(\lambda))_{\lambda}=0$
for $\lambda=\frac{p-1}{2}$. Thus, the injection is an isomorphism in  
all cases and
the proposition follows.
\end{proof}

\section{Extensions Between Simple Modules}\Label{S:simple}

\subsection{} \Label{SS:extensionthms} The following theorem  
describes first extension
groups between all simple modules for the restricted simple modules  
of $W(1,1)$.
The proof of the theorem will occupy the remainder of Section~\ref 
{S:simple}.

\begin{thm}[A] \Label{T:ExtLLp=5} Let ${\mathfrak g}=W(1,1)$ and $p=5 
$. Then
\begin{itemize}
\item[(a)] $\Ext^{1}_{u(\fg)}(L(\mu),L(\la)) \cong k$ if
  \begin{itemize}
  \item[(i)] $\la-\mu \equiv 2, 3 \pmod p,\ 1\le\mu\le p-2,\ 1\le\la 
\le p-1$, or
  \item[(ii)] $(\mu,\la)=(0,1), (p-2,0), (p-1,2)$, or $(p-1,3)$
  \end{itemize}
\item[(b)] $\Ext^{1}_{u(\fg)}(L(\mu),L(\la)) \cong k\oplus k$ if $\{\mu,\la\} = \{0,p-1\}$
\item[(c)] $\Ext^{1}_{u(\fg)}(L(\mu),L(\la)) =0$ otherwise.
\end{itemize}
\end{thm}

\begin{thm}[B] \Label{T:ExtLLp>5} Let ${\mathfrak g}=W(1,1)$ and $p 
\geq 7$. Then
\begin{itemize}
\item [(a)] $\Ext^{1}_{u(\fg)}(L(\mu),L(\la)) \cong k$ if
  \begin{itemize}
  \item[(i)] $\la-\mu \equiv 2, 3, 4 \pmod p,\ 1\le\mu\le p-2,\ 1\le 
\la\le p-1$, or
  \item[(ii)] $(\mu,\la)=(0,1), (p-2,0), (p-1,2), (p-1,3)$, or $ 
(p-1,4)$, or
  \item[(iii)] $2\la^{2}-10\la+3\equiv 0 \pmod p,\ \la-\mu\equiv 6  
\pmod p,\ 1\le \mu,\la \le p-2$
  \end{itemize}
\item[(b)] $\Ext^{1}_{u(\fg)}(L(\mu),L(\la)) \cong k\oplus k$ if $\{\mu,\la\} = \{0,p-1\}$
\item[(c)] $\Ext^{1}_{u(\fg)}(L(\mu),L(\la)) =0$ otherwise.
\end{itemize}
\end{thm}

The $\text{Ext}^{1}$-quivers between simple modules
for $p=5,7$ are shown in Figure \ref{Fig:Ext1simple}.

\psfrag{0}{{\footnotesize$0$}}
\psfrag{1}{{\footnotesize$1$}}
\psfrag{2}{{\footnotesize$2$}}
\psfrag{3}{{\footnotesize$3$}}
\psfrag{4}{{\footnotesize$4$}}
\psfrag{5}{{\footnotesize$5$}}
\psfrag{6}{{\footnotesize$6$}}
\psfrag{7}{{\footnotesize$7$}}
\psfrag{8}{{\footnotesize$8$}}
\psfrag{9}{{\footnotesize$9$}}
\psfrag{10}{{\footnotesize$10$}}
\psfrag{p=5}{{$p=5$}}
\psfrag{p=7}{{$p=7$}}
\begin{figure}[h]
\centering
\includegraphics[height=.25\textheight]{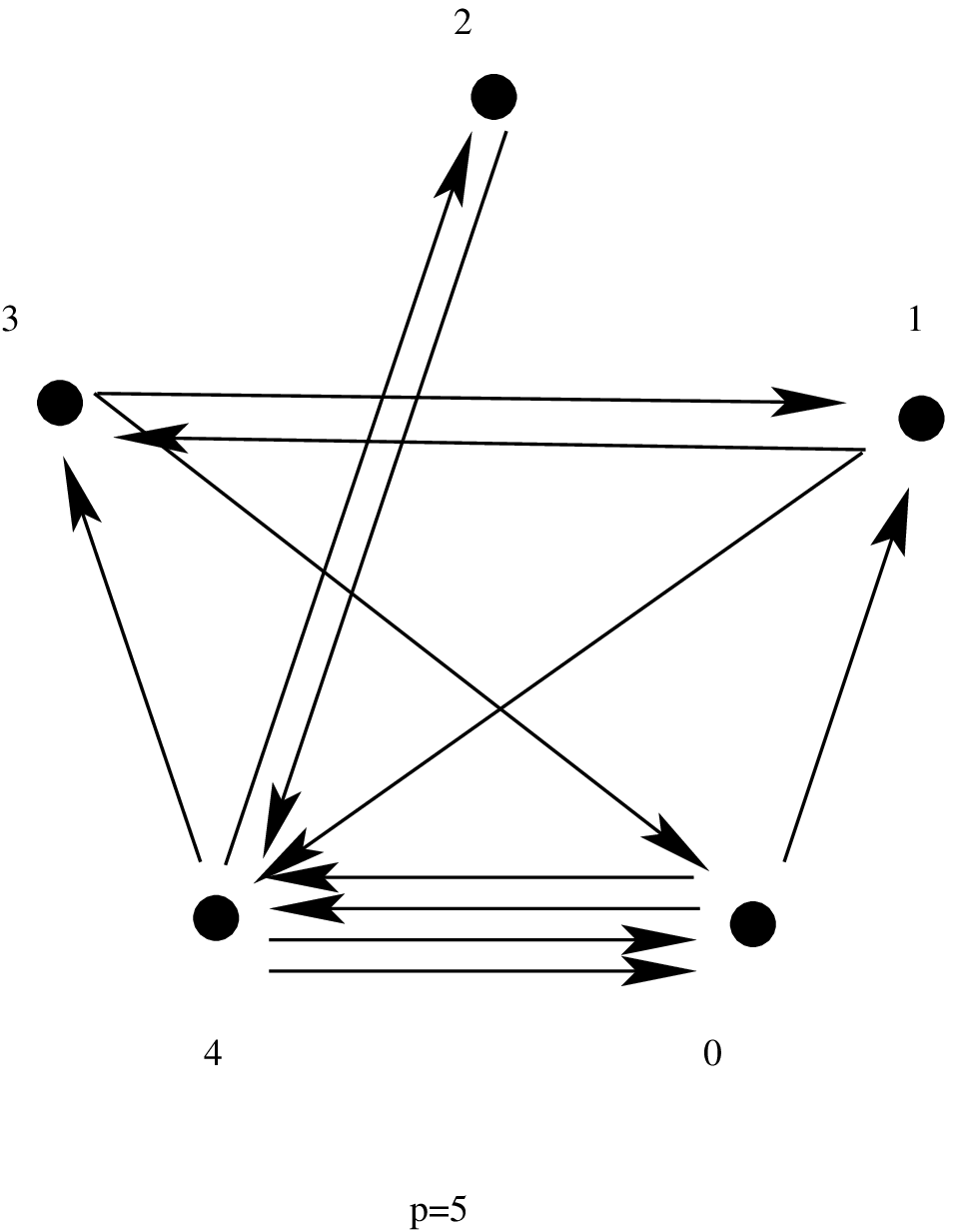}
\hspace{.5in}
\includegraphics[height=.3\textheight]{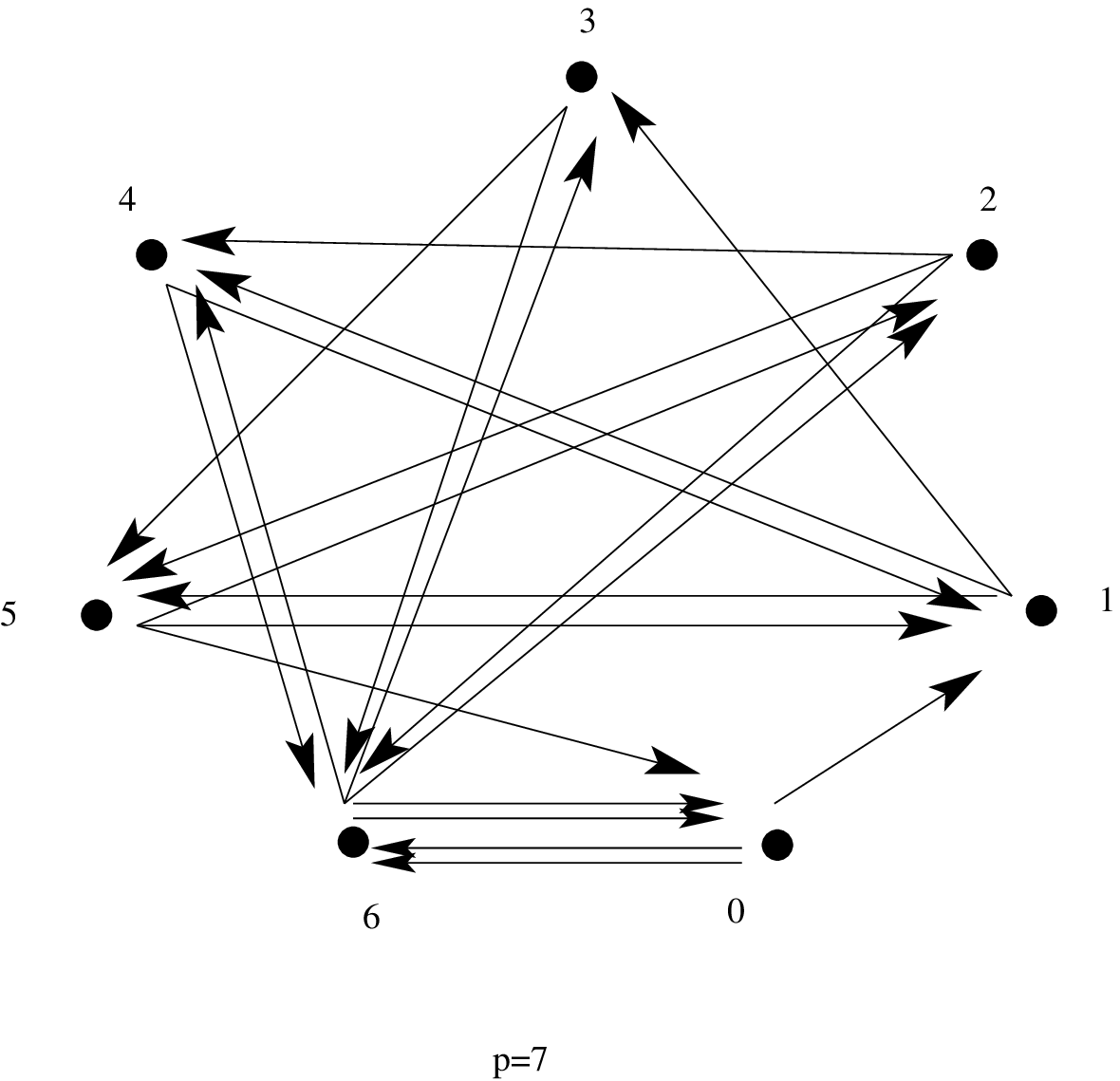}
\caption{$\Ext^1$-quivers for $L(\la)$ ($p=5,7$)} \Label{Fig:Ext1simple}
\end{figure}

\subsection{Extensions not involving $L(0)$ or $L(p-1)$} \Label 
{SS:generic}
Since $L(\la)=Z^{+}(\la)$ for $1\le\la\le p-2$, we have that
\begin{equation} \Label{E:genericLext}
\Ext^{1}_{u(\fg)}(L(\mu),L(\la)) = \Ext^{1}_{u(\fg)}(Z^{+}(\mu),Z^{+} 
(\la)) \text{\quad for \quad} 1\le \mu,\la \le p-2.
\end{equation}
These extensions were computed in Section \ref{SS:Vermaextension}. So  
it remains to compute the extensions
involving $L(0)\cong k$ or $L(p-1)$.

\subsection{Extensions involving $L(0)$} \Label{SS:L0}
As in the proof of Proposition \ref{P:Derivation}, we have
\begin{equation} \Label{E:Ext1asH1}
\begin{aligned}
\Ext^{1}_{u(\fg)}(\Zmu,k)  &\cong \H^{1}(u(\fn^{+}),k)_{\mu} \\
  &\cong \H^{1}(\fn^{+},k)_{\mu} \\
&\cong
  \begin{cases}
  k, &\mu=p-1,\, p-2,\\
  0, &\text{otherwise}.
  \end{cases}
\end{aligned}
\end{equation}
For the first two isomorphisms we are using \cite[(2.7) Proposition] 
{Fe} and the fact that
$p$-powers of elements in $\fn^{+}$ are zero. The last isomorphism  
follows because
$\H^{1}(\fn^{+},k)$ is spanned by the classes of $e_{1}^{*}$ and $e_ 
{2}^{*}$,
with weights $-1$ and $-2$. In particular
\begin{equation} \Label{E:LmuL0}
\Ext^{1}_{u(\fg)}(L(\mu),L(0))=
\begin{cases}
  0, &1\le \mu \le p-3,\\
  k, &\mu=p-2.
  \end{cases}
\end{equation}
We postpone the case $\mu=p-1$ to Section \ref{SS:L0pm1}.

We can obtain extensions with $L(0)$ in the first variable by using  
the identity
\begin{equation} \Label{E:extduality}
\Ext^{1}_{u(\fg)}(L(\mu),L(\la)) \cong \Ext^{1}_{u(\fg)}(L(\la)^{*},L 
(\mu)^{*}).
\end{equation}
Note that $L(0)$ and $L(p-1)$ are isomorphic to their own duals,  
since they are the only simple modules of dimension 1 and $p-1$,  
respectively. For $\mu\ne 0,\ p-1$, $L(\mu)$ has dimension $p$ and  
weights (from highest to lowest)
$$
\mu, \mu-1, \dots, \mu-p+1.
$$
Therefore $L(\mu)^{*}$ has weights (from highest to lowest)
$$
p-\mu-1, p-\mu-2, \dots, -\mu.
$$
In conclusion,
\begin{equation} \Label{E:simpleduals}
L(\mu)^{*}\cong
\begin{cases}
L(\mu), &\mu=0,\, p-1,\\
L(p-1-\mu), &\text{otherwise}.
\end{cases}
\end{equation}

Combining \eqref{E:LmuL0}, \eqref{E:extduality}, and \eqref 
{E:simpleduals} we deduce
\begin{equation} \Label{E:L0Lmu}
\Ext^{1}_{u(\fg)}(L(0),L(\la))=
\begin{cases}
  k, &\la=1,\\
  0, &2\le \la \le p-2.
  \end{cases}
\end{equation}

\subsection{Extensions involving $L(p-1)$} \Label{SS:Lpm1}
Consider the short exact sequence
\begin{equation} \Label{E:sesZ0}
0\to L(p-1)\to Z^{+}(0) \to k \to 0.
\end{equation}
Given any $\mu$ we have a corresponding long exact sequence
\begin{multline} \Label{E:lesZ0}
\cdots \to \Hom_{u(\fg)}(\Zmu,k) \to \Ext^{1}_{u(\fg)}(\Zmu,L(p-1))  
\to \Ext^{1}_{u(\fg)}(\Zmu,Z^{+}(0)) \\
  \to \Ext^{1}_{u(\fg)}(\Zmu,k) \to \cdots
\end{multline}
Assume that $\mu\ne 0,\ p-2,\ p-1$. Then the first term is 0, and the  
last term is 0 by \eqref{E:LmuL0}.
So the middle arrow is an isomorphism. Using the results of Section  
\ref{SS:Vermaextension} we deduce
\begin{equation} \Label{E:LmuLpm1}
\Ext^{1}_{u(\fg)}(L(\mu),L(p-1))=
\begin{cases}
  k, &p-5 \le \mu \le p-3\ \ (1\le \mu\le 2 \text{ if } p=5),\\
  0, &1\le \mu \le p-6.
  \end{cases}
\end{equation}
Dually we obtain
\begin{equation} \Label{E:Lpm1Lmu}
\Ext^{1}_{u(\fg)}(L(p-1),L(\la))=
\begin{cases}
  k, &2 \le \la \le 4\ \ (2\le \la\le 3 \text{ if } p=5),\\
  0, &5\le \la \le p-2.
  \end{cases}
\end{equation}

Next consider the short exact sequence
\begin{equation} \Label{E:sesZpm1}
0\to k\to Z^{+}(p-1) \to L(p-1) \to 0.
\end{equation}
We have a corresponding long exact sequence
\begin{multline} \Label{E:lesZpm1}
\cdots \to \Hom_{u(\fg)}(\Zmu,L(p-1)) \to \Ext^{1}_{u(\fg)}(\Zmu,k)  
\to \Ext^{1}_{u(\fg)}(\Zmu,Z^{+}(p-1)) \\
  \to \Ext^{1}_{u(\fg)}(\Zmu,L(p-1)) \to \Ext^{2}_{u(\fg)}(\Zmu,k)  
\to \cdots
\end{multline}

Take $\mu=p-2$ (the cases $\mu=0,\ p-1$ will be treated in the next  
subsection). We want to determine the fourth term of \eqref 
{E:lesZpm1}. The first term is 0, the second term is $k$ by \eqref 
{E:LmuL0}, and the third term is $k$ by Section \ref 
{SS:Vermaextension}. Arguing similarly to \eqref{E:Ext1asH1}, we have
$$
\Ext^{2}_{u(\fg)}(\Zmu,k) \cong \H^{2}(u(\fn^{+}),k)_{\mu} .
$$
There is a spectral sequence constructed in \cite[(1.3) Proposition] 
{FrPa:87}
(since $p>2$)
$$
E_{1}^{2i,j}=S^{i}((\fn^{+})^{*})^{(1)} \otimes \H^{j}(\fn^{+},k)  
\Rightarrow \H^{2i+j}(u(\fn^{+}),k).
$$
Only two terms contribute to $\H^{2}(u(\fn^{+}),k)$. The term with  
$i=1,\ j=0$ is $S^{1}((\fn^{+})^{*})^{(1)} \cong ((\fn^{+})^{*})^{(1)} 
$, but all its weights are multiples of $p$, so there is no  
contribution to the $p-2$ weight space from this term. The term with  
$i=0,\ j=2$ is $H^{2}(\fn^{+},k)$, a subquotient of $\Lambda^{2}(\fn^ 
{+})^{*}$. In the case $p=5$, the $p-2$ weight space is zero.  For  
$p>5$, an ordered basis for the $p-2$ weight space is
\begin{equation} \Label{E:wedge2basis}
e_{4}^{*}\wedge e_{p-2}^{*},\ e_{5}^{*}\wedge e_{p-3}^{*}, \dots, e_ 
{\frac{p-3}{2}}^* \wedge e_{\frac{p-1}{2}}^* .
\end{equation}

Suppose $z$ is a nonzero linear combination of these terms, and let  
$e_{r}^{*}\wedge e_{p-r+2}^{*}$ be the first nonzero term appearing.  
Then $d_{2}$ applied to that term will involve a nonzero multiple $e_ 
{1}^{*}\wedge e_{r-1}^{*}\wedge e_{p-r+2}^{*}$, but this term cannot  
appear in $d_{2}$ of any later basis vector in \eqref{E:wedge2basis}.  
Thus $z$ cannot belong to the kernel of $d_{2}$, and so $\H^{2}(\fn^ 
{+},k)_{p-2}=0$ as well. We conclude that $\H^{2}(u(\fn^{+}),k)_{\mu} 
=0$.

It follows from \eqref{E:lesZpm1} that
\begin{equation}
\Ext^{1}_{u(\fg)}(L(p-2),L(p-1))=0
\end{equation}
and dually that
\begin{equation}
\Ext^{1}_{u(\fg)}(L(p-1),L(1))=0.
\end{equation}

\subsection{Extensions involving only $L(0)$ and $L(p-1)$} \Label 
{SS:L0pm1}
There remains to compute the extensions between $L(0)$ and $L(p-1)$.  
Note that by
\cite[Thm. 5.4]{LN} there are no self-extensions between simple  
modules in this category.
For $L(0)\cong k$ this can be seen directly, because
$$\Ext^{1}_{u({\fg})}(k,k) \cong \H^{1}(u(\fg),k) \cong \fg/[\fg,\fg]_ 
{p} = 0$$
where $[\fg,\fg]_{p}$ is the subalgebra generated by commutators and  
$p$th powers.
In order to compute $\Ext^{1}_{u(\fg)}(L(p-1),L(0))$ we can use  
another long exact sequence associated to \eqref{E:sesZpm1}:
\begin{multline} \Label{E:les2Zpm1}
\cdots \to \Hom_{u(\fg)}(Z^{+}(p-1),k) \to \Hom_{u(\fg)}(k,k) \to  
\Ext^{1}_{u(\fg)}(L(p-1),k) \\
  \to \Ext^{1}_{u(\fg)}(Z^{+}(p-1),k) \to \Ext^{1}_{u(\fg)}(k,k) \to  
\cdots
\end{multline}
The first term is 0, the second term is $k$, the fourth term is $k$  
by \eqref{E:Ext1asH1}, and the last term is 0 by the fact
just mentioned. Hence the middle term must be $k\oplus k$.  That is,
\begin{equation}
\Ext^{1}_{u(\fg)}(L(p-1),L(0))\cong k\oplus k
\end{equation}
and dually
\begin{equation}
\Ext^{1}_{u(\fg)}(L(0),L(p-1))\cong k\oplus k.
\end{equation}

\section{$\text{Ext}^{1}$ between simple non-restricted modules}

 From the definition of the height, one sees that $\chi=0$ if and  
only if $r(\chi)=-1$.
This case coincides with considering $u({\mathfrak g})$-modules.  
Since we have
already dealt with extensions in this setting we can now consider the  
cases when
$0\leq r(\chi) \leq p-1$.

\subsection{$r(\chi)=0$:} In this case there are $p-1$ non-isomorphic
simple $u({\mathfrak g},\chi)$-modules \cite[Hauptsatz $2^{\prime}$] 
{Ch} each
of dimension $p$. These simple modules can be constructed as follows  
(cf.\
\cite[Section 1.3]{FN}).
Since $\chi({\mathfrak b}^{+})=0$ we have $u({\mathfrak b}^{+},\chi)=u({\mathfrak b}^{+})$. 
The simple $u({\mathfrak b}^{+})$ modules are
one-dimensional and given by weights $\lambda\in {\mathbb F}_{p}$. Let
$L(\lambda)=u({\mathfrak g},\chi)\otimes_{u({\mathfrak b}^{+})}\lambda$. Then
$L(\lambda)$ is a simple $u({\mathfrak g},\chi)$-module. Moreover
$\{L(\lambda):\ \lambda=0,1,2,\dots,p-2 \}$ forms a complete set of
non-isomorphic $u({\mathfrak g},\chi)$-modules. We remark that
$L(0)\cong L(p-1)$ in our construction.

The action of $u({\mathfrak n}^{+})$ on $L(\la)$ is given by
\eqref{moduleaction}. Consequently, we can use the calculations in  
Section~\ref{SS:VermaDercomp}
to compute $\text{H}^{1}({\mathfrak n}^{+},L(\lambda))$ for $ 
\lambda=0,1,\dots,p-2$. The same
argument as in Proposition~\ref{P:Derivation} demonstrates that
$$\text{Ext}^{1}_{u({\mathfrak g},\chi)}(L(\mu),L(\lambda))\cong
\text{H}^{1}(u({\mathfrak n}^{+}),L(\lambda))_{\mu}\cong \text{H}^{1} 
({\mathfrak n}^{+},L(\lambda))_{\mu}.$$
Consequently, the extensions between simple modules for $u({\mathfrak  
g},\chi)$ can be obtained by removing
  the node corresponding to $Z^{+}(p-1)$, and all edges attached to  
it, in  the $\text{Ext}^{1}$-quiver
for the Verma modules given in Theorem ~\ref{Thm:ZExt}.

\begin{thm} \Label{Thm:r=0Ext} Let ${\mathfrak g}=W(1,1)$ and $r(\chi) 
=0$. Assume $0 \le \la,\mu \le p-2$.
\begin{itemize}
\item[(a)] For $p=5$,
$$
\Ext^1_{u(\fg,\chi)}(L(\mu),L(\la)) \cong \left\{ \begin{array}{ll}
k & \mbox{if} \ \la - \mu=2,3 \\
& \mbox{or} \ (\mu, \la) = (0,0), (1,0), (0,1),   \\
0 & \mbox{otherwise}.
\end{array} \right.
$$
\item[(b)] For $p \geq 7$,
$$\Ext^1_{u(\fg,\chi)}(L(\mu),L(\la))\cong k$$
if
\begin{itemize}
\item[(i)] $\la - \mu=2,3,4$;
\item[(ii)] $(\mu, \la)=(0,0), (0,1), (-5,0) $;
\item[(iii)] $\lambda$ is such that $2 \la^2-10 \la +3=0$ and $\la- 
\mu=6$.
\end{itemize}
Otherwise, $\Ext^1_{u(\fg,\chi)}(L(\mu),L(\la))= 0$.
\end{itemize}
\end{thm}

\subsection{$r(\chi)=1$:} From \cite[Hauptsatz $2^{\prime}$]{Ch}, we  
have
$p$ isomorphism classes of simple modules each of which has
dimension $p$. We may assume that we
are working with the character $\chi$ where $\chi(e_{i})=0$ for $i 
\neq 0$
and $\chi(e_{0})=1$ because all height $1$ characters are conjugate
to this $\chi$ under the automorphism group of $\fg$. Fix $\xi$ such  
that $\xi^p-\xi-1=0$.
The simple modules will be in bijective correspondence with
$\Lambda(\chi)=\{ \lambda+\xi: \lambda\in {\mathbb F}_{p}\}$. If
$\lambda+\xi\in \Lambda(\chi)$, we can construct $L(\lambda)=u 
({\mathfrak g},\chi)
\otimes_{u({\mathfrak b}^{+},\chi)} (\lambda+\xi)$ where $\lambda+\xi 
$ is
a one-dimensional $u({\mathfrak t},\chi)$-module which is inflated to
$u({\mathfrak b}^{+},\chi)$ by letting $u({\mathfrak n}^{+})$ act  
trivially.

In order to apply our techniques, we can construct a basis for the
$u({\mathfrak b}^{+})$-module
$$
L(\lambda)\otimes -\xi = \langle m_0, \ldots, m_{p-1} \rangle, \quad  
0 \leq \lambda <p
$$
with relations
\begin{eqnarray*}
e_0 . m_j &=& (\lambda + j + 1) m_j \\
e_k . m_j &=& (j+k+1+(k+1)(\lambda+ \xi)) m_{j+k}.
\end{eqnarray*}

\begin{thm} Let ${\mathfrak g}=W(1,1)$ and $r(\chi)=1$.
\begin{itemize}
\item[(a)] For $p =5$,
$$
\Ext^1_{u(\fg,\chi)}(L(\mu), L(\lambda))  \cong   \left\{
\begin{array}{ll}
k & \mbox{if} \ \lambda - \mu = 2,3, \\
0 & \mbox{otherwise}.
\end{array} \right.
$$
\item[(b)] For $p \geq 7$,
$$
\Ext^1_{u(\fg,\chi)}(L(\mu), L(\lambda))  \cong   \left\{
\begin{array}{ll}
k & \mbox{if} \ \lambda - \mu = 2,3,4, \\
0 & \mbox{otherwise}.
\end{array} \right.
$$
\end{itemize}
\end{thm}
\begin{proof} Following the proof of Theorem \ref{Thm:ZExt}, we will  
first compute
$\H^{1}({\mathfrak n}^{+},L(\lambda)\otimes -\xi)_{\mu}$. Note that $L 
(\lambda)\otimes -\xi$
is a restricted module for ${\mathfrak b}^{+}$, so the cohomology is  
a restricted ${\mathfrak t}$-module.
Write $\mu=\la+j$ for some $0 \leq j \leq p-1$.
For $d \in \Der(\fn^+, L(\la)\otimes -\xi)_{\la+j}$, $d(e_k) = a_k m_ 
{\overline{j+k-1}}$ for
some $a_k \in k$, as before.

We begin by analyzing the inner derivations of $L(\la)\otimes -\xi$.   
For $j>0$, the inner
derivation $d_j$ defined by $d_j(e_k)=e_k . m_{j-1}$ is nonzero (of  
weight $\la+j$) since
$e_1 . m_{j-1} = (j+1+2(\la+\xi)) m_j  \neq 0$.
Therefore,
\begin{equation}
\Inn (\fn^+, L(\la)\otimes -\xi)_{\la+j} = \left\{
\begin{array}{ll}
k & \mbox{if} \  1 \leq j \leq p-1, \\
0 &\mbox{if} \ j=0.
\end{array} \right. 
\end{equation}

We now compute  $\Der (\fn^+, L(\la)\otimes -\xi)_{\la+j}$, using  
arguments similar to the proof of
Theorem \ref{Thm:ZExt}.  In particular, the calculations in the proof of
Theorem \ref{Thm:ZExt} can be used in this setting by replacing $\la$  
with $\la+\xi$.

We first consider $p>7$.  As in the restricted case, for $p-4 \leq j  
\leq p-2$, we have
$\dim \Der (\fn^+, L(\la)\otimes -\xi)_{\la+j} = 2$.  Thus
$\H^1({\mathfrak n}^{+},L(\lambda)\otimes -\xi)_{\la+j} \cong k$.

Suppose $0 <j \leq p-7$.  Since $\Inn (\fn^+, L(\la)\otimes -\xi)_{\la 
+j}=k$, it is enough to check if
$$\dim \Der (\fn^+, L(\la)\otimes -\xi)_{\la+j} = 2.$$  If this is  
true, we must have
$$
0=p_{1,5}(j, \la+\xi)=p_{2,5}(j, \la+\xi) =p_{1,7}(j, \la+\xi)=p_{2,7} 
(j, \la+\xi),
$$
as in \eqref{Eq:Zmiddlecase}.  Since $j(j+1)(j+\la+\xi) \neq 0$ (cf.\  
\eqref{Eq:Zmiddlecase2}),
this condition fails to hold.  Thus, $\dim \Der (\fn^+, L(\la)\otimes  
-\xi)_{\la+j} = 1$ for $0 <j \leq p-7$,
and so $\H^1({\mathfrak n}^{+},L(\lambda)\otimes -\xi)_{\lambda+j}=0$.

For $j=p-5$, observe that $p_{1,5}(p-5, \la+\xi)$, $p_{2,5}(p-5, \la+ 
\xi) \neq 0$ from
\eqref{Eq:ZExt5}.  Similarly, for $j=p-6$ we have that  $p_{1,6}(p-6,  
\la+\xi)$, $p_{2,6}(p-6, \la+\xi) \neq 0$.   In this case, we use the 
fact that  $3-10(\la+\xi)+2(\la+\xi)^2
\neq 0$ since $\la+\xi$ is not a root of any polynomial over $ 
{\mathbb F}_p$ of degree less than
$p$.  For $j=p-1$, note that $p_{2,5}(p-1, \la+\xi)$, $p_{2,7}(p-1, \la+ 
\xi) \neq 0$.  Therefore,
$\dim \Der (\fn^+, L(\la)\otimes -\xi)_{\la+j} = 1$ for $j=p-5$, $p-6$, 
$p-1$, which implies that $\H^1({\mathfrak n}^{+},L(\lambda)\otimes -\xi)_ 
{\lambda+j}=0$.

Suppose $j=0$, and let $d \in \Der (\fn^+, L(\la)\otimes -\xi)_{\la}$.
Then \eqref{Eq:j=0a}, \eqref{Eq:j=0b}, and \eqref{Eq:j=0c} become
\begin{eqnarray*}
0&=& d([e_2,e_{p-2}])  =  \frac{2}{p-4}(\la+\xi)(3(\la+\xi-1)a_{p-3}+3 
(\la+\xi+1)a_1-2a_2)m_{p-1};\\
0&=& d([e_1, e_{p-2}]) =\frac{2}{p-4}((\la+\xi-1)(2(\la+\xi)-1)a_{p-3} 
+(\la+\xi+1)(2(\la+\xi)-3)a_1)m_{p-2};\\
0&=&d([e_2, e_{p-3}])=((3(\la+\xi)-1)a_{p-3}+(2(\la+\xi)+1)a_2)m_{p-2}.
\end{eqnarray*}
These equations have a common solution only if
$$
0=-3(\la+\xi)(\la+\xi+1)(20(\la+\xi)^2-7(\la+\xi)-11) .
$$
Since there is no $\la\in{\mathbb F}_{p}$ that satisfies this  
equation, we have $\Der (\fn^+, L(\la)\otimes -\xi)_{\la}=0$.

For $p=7$, we may carry over the arguments from the proof of  Theorem  
\ref{Thm:ZExt} to
obtain the result.  If $p=5$, $\dim \Der (\fn^+, L(\la)\otimes -\xi)_ 
{\la+j}$ can be computed as above for
$2 \leq j \leq 3$.  For $j=0,1,4$, the proof of Theorem  \ref 
{Thm:ZExt}  implies
$\Der (\fn^+, L(\la)\otimes -\xi)_{\la+j}=0$.

One can apply the same argument given in Proposition~\ref 
{P:Derivation} to show that
$$
\Ext^1_{u(\fg,\chi)}(L(\mu), L(\lambda))  \cong \H^{1}(u({\mathfrak n} 
^{+}),L(\lambda)\otimes -\xi)_{\mu}\hookrightarrow \H^{1}({\mathfrak n}^{+},L(\lambda)\otimes -\xi)_{\mu}.
$$
Now from the proof of Proposition~\ref {P:Derivation}, if $\phi$ 
is in $\H^{1}({\mathfrak n}^{+},L(\lambda)\otimes -\xi)_{\mu}$, but 
not in $\H^{1}(u({\mathfrak n}^{+}),L(\lambda)\otimes -\xi)_{\mu}$ 
then $e_{1}^{p-1}.m_{0}\neq 0$, and $\phi(e_{1})=m_{0}$. This can occur for any $\lambda$, but
$e_{0}.m_{0}=(\lambda+1)m_{0}$, thus the weight of $\phi$ is $-1+ 
(\lambda+1)=\lambda$. This contradicts our preceding calculation because $\H^{1}(n^{+},L 
(\lambda)\otimes -\xi)_{\lambda}=0$.
Therefore,
$\H^{1}(u({\mathfrak n}^{+}),L(\lambda)\otimes -\xi)_{\mu}\cong
\H^{1}(n^{+},L(\lambda)\otimes -\xi)_{\mu}$ for all $\lambda,\mu\in  
{\mathbb F}_{p}$.
\end{proof}

\subsection{$1<r(\chi)<p-1$:} From \cite[Hauptsatz 1]{Ch} there is a  
unique simple
$u({\mathfrak g},\chi)$-module denoted by $L$. Let $r=r(\chi),\ s:= 
\lfloor\frac{r}{2}\rfloor$ and $k_{\chi}$
be the unique one dimensional simple $u({\mathfrak g}_{s},\chi)$- 
module. Then the simple
module $L\cong u({\mathfrak g},\chi)\otimes_{u({\mathfrak g}_{s}, 
\chi)}k_{\chi}$.
Let $S:=u({\mathfrak b}^{+},\chi)\otimes_{u({\mathfrak g}_{s},\chi)}k_ 
{\chi}$. The
module $S$ is a simple $u({\mathfrak b}^{+},\chi)$-module because $L 
\cong
u({\mathfrak g},\chi)\otimes_{u({\mathfrak b}^{+},\chi)}S$. Moreover,  
$S$ is
the unique simple $u({\mathfrak b}^{+},\chi)$-module. Viewing $\chi 
\in ({\mathfrak b}^{+})^{*}$ then
we have a bilinear form $\beta_{\chi}$ on ${\mathfrak b}^{+}$ defined by
${\beta}_{\chi}(x,y)=\chi([x,y])$. The radical of this form will be  
denoted by
$\text{rad}({\beta}_{\chi})$.

In the case when $r$ is even, there is only one conjugacy class of  
characters $\chi\in\fg^{*}$ of height $r$ under the automorphism group
of ${\mathfrak g}$. A representative of the class is given by $\chi(e_ 
{r-1})=1$ and
$\chi(e_{j})=0$ for $j\neq r-1$ \cite[Theorem 3.1(a)]{FN}. So without  
loss of generality
we may assume for $r$ even that $\text{rad}({\beta}_{\chi})= 
{\mathfrak g}_{r}$, which is an ideal in
${\mathfrak b}^{+}$. Moreover, $\chi({\mathfrak g}_{r})=0$ in this  
setting. For $r$ odd there are infinitely
many conjugacy classes of characters of height $r$ (cf.\ \cite 
[Theorem 3.1(b)]{FN}). The following theorem is an
application of a result proved by Farnsteiner \cite[Theorem 3.8]{Fa}  
which provides a description
of the $\text{Ext}^1$-quiver in $u({\mathfrak b}^{+},\chi)$.

\begin{thm}[A] Suppose that $\operatorname{rad}(\beta_{\chi})$ is an  
ideal of
${\mathfrak b}^{+}$. Then
\begin{itemize}
\item[(a)] $\operatorname{Ext}^{1}_{u({\mathfrak b}^{+},\chi)}(S,S)\cong
\operatorname{H}^{1}(u(\operatorname{rad}(\beta_{\chi})),k)$;
\item[(b)] If $r$ is even then $\operatorname{Ext}^{1}_{u({\mathfrak  
b}^{+},\chi)}(S,S)\cong
({\mathfrak g}_{r}/[{\mathfrak g}_{r},{\mathfrak g}_{r}])^{*}$.
\end{itemize}
\end{thm}

\begin{proof} Part (a) follows directly from \cite[Theorem 3.8(2)] 
{Fa} and the
fact that $\operatorname{Ext}^{1}_{u(\operatorname{rad}(\beta_{\chi}))}
(k_{\chi},k_{\chi})\cong \text{H}^{1}(u(\operatorname{rad}(\beta_ 
{\chi})),k)$.
For part (b), we use the fact that $\text{rad}({\beta}_{\chi})= 
{\mathfrak g}_{r}$. Moreover,
since $({\mathfrak g}_{r})^{[p]}=0$ it follows that
$$\operatorname{H}^{1}(u(\operatorname{rad}(\beta_{\chi})),k)\cong
\operatorname{H}^{1}(\operatorname{rad}(\beta_{\chi}),k)\cong
({\mathfrak g}_{r}/[{\mathfrak g}_{r},{\mathfrak g}_{r}])^{*}.$$
\end{proof}

We next prove that one can find a lower bound for $\operatorname{Ext}^ 
{1}_{u({\mathfrak g},\chi)}(L,L)$
in terms of $\operatorname{Ext}^{1}_{u({\mathfrak b}^{+},\chi)}(S,S)$.

\begin{prop}[B]  Let $1<r(\chi)<p-1$. Then
$$\dim \operatorname{Ext}^{1}_{u({\mathfrak g},\chi)}(L,L)\geq \dim
\operatorname{Ext}^{1}_{u({\mathfrak b}^{+},\chi)}(S,S)-1.$$
\end{prop}

\begin{proof} First observe by Frobenius reciprocity,
$$\text{Ext}^{1}_{u({\mathfrak g},\chi)}(L,L)\cong \text{Ext}^{1}_{u 
({\mathfrak b}^{+},\chi)}(S,L).$$
The module $L|_{u({\mathfrak b}^{+},\chi)}$ has $p$ composition  
factors isomorphic to $S$.
By a direct computation, one can prove that $L|_{u({\mathfrak b}^{+}, 
\chi)}$ is uniserial.
We have a short exact sequence of $u(\fb^{+},\chi)$-modules
$$0\rightarrow S \rightarrow L \rightarrow N \rightarrow 0$$
with $N$ uniserial of length $p-1$.
Since $\text{Hom}_{u({\mathfrak b}^{+},\chi)}(S,S)\cong \text{Hom}_{u 
({\mathfrak b}^{+},\chi)}(S,L)
\cong k$, the long exact sequence induced from the short exact  
sequence is
$$0\rightarrow \text{Hom}_{u({\mathfrak b}^{+},\chi)}(S,N)
\rightarrow \text{Ext}^{1}_{u({\mathfrak b}^{+},\chi)}(S,S)\rightarrow
\text{Ext}^{1}_{u({\mathfrak b}^{+},\chi)}(S,L)\rightarrow \dots$$
But, $\text{Hom}_{u({\mathfrak b}^{+},\chi)}(S,N)\cong k$ and from  
the sequence
$\text{Ext}^{1}_{u({\mathfrak b}^{+},\chi)}(S,S)/k$ embeds into
$\text{Ext}^{1}_{u({\mathfrak b}^{+},\chi)}(S,L)$.

\end{proof}

Suppose that $\chi\in ({\mathfrak b}^{+})^{*}$ and
$\chi(\operatorname{rad}(\beta_{\chi}))=0$ (recall that we may assume  
this when $r$ is
even). Then one can regard $\chi\in ({\mathfrak b}^{+}/\operatorname 
{rad}(\beta_{\chi}))^{*}$ and
form the reduced enveloping algebra $u({\mathfrak b}^{+}/\operatorname 
{rad}(\beta_{\chi}),\chi)$.
Under this assumption, we can provide an interpretation of $\text{Ext} 
^{n}_{u({\mathfrak g},\chi)}(L,L)$.

\begin{thm}[C] Suppose that $\operatorname{rad}(\beta_{\chi})$ is an  
ideal of ${\mathfrak b}^{+}$ and
$\chi(\operatorname{rad}(\beta_{\chi}))=0$. Then
\begin{itemize}
\item[(a)] $S^{\operatorname{rad}(\beta_{\chi})}=S$.
\item[(b)] The algebra $u({\mathfrak b}^{+}/\operatorname{rad}(\beta_ 
{\chi}),\chi)$ is semisimple.
\item[(c)] The module $S$ is a simple and projective $u({\mathfrak b}^ 
{+}/\operatorname{rad}(\beta_{\chi}),\chi)$-module.
\item[(d)] $\operatorname{Ext}^{n}_{u({\mathfrak g},\chi)}(L,L)\cong
\operatorname{Hom}_{u({\mathfrak b}^{+}/\operatorname{rad}(\beta_ 
{\chi}),\chi)}
(S,\operatorname{H}^{n}(u(\operatorname{rad}(\beta_{\chi})),L))$ for  
all $n$.
\end{itemize}
\end{thm}

\begin{proof} (a) We apply the constructions in the proof of \cite 
[Theorem 3.8]{Fa}. In the proof it
is shown that
$$S\cong u({\mathfrak g}_{s},\chi)\otimes_{u(\operatorname{rad}(\beta_ 
{\chi}))} k.$$
Since
$\operatorname{rad}(\beta_{\chi})$ is an ideal of ${\mathfrak b}^{+} 
$, it is also an ideal of
$\fg_{s}$. By using the isomorphism above, one sees that $S^ 
{\operatorname{rad}(\beta_{\chi})}=S$.

(b), (c) Since $S^{\operatorname{rad}(\beta_{\chi})}=S$, one can  
regard $S$ as a
$u({\mathfrak b}^{+}/\operatorname{rad}(\beta_{\chi}),\chi)$-module.  
It is simple as
$u({\mathfrak b}^{+}/\operatorname{rad}(\beta_{\chi}),\chi)$-module  
because it is
simple as $u({\mathfrak b}^{+},\chi)$-module.

Next observe that $\dim u({\mathfrak b}^{+}/\operatorname{rad}(\beta_ 
{\chi}),\chi)=p^{a}$
where $a=\dim {\mathfrak b}^{+}-\dim \operatorname{rad}(\beta_{\chi})$
Recall that $\dim S=p^{b}$ where $b=\dim {\mathfrak b}^{+}-\dim  
{\mathfrak g}_{s}$. The isomorphism
above shows that $\dim S=p^{z}$ where $z=\dim {\mathfrak g}_{s}-\dim  
\operatorname{rad}(\beta_{\chi})$.
Therefore,
$$\dim {\mathfrak b}^{+}-\dim {\mathfrak g}_{s}=\dim {\mathfrak g}_ 
{s}-\dim \operatorname{rad}(\beta_{\chi}).$$
In other words
$$\dim {\mathfrak b}^{+}-\dim \operatorname{rad}(\beta_{\chi})=2(\dim  
{\mathfrak g}_{s}-
\dim \operatorname{rad}(\beta_{\chi})).$$
This proves that
$$\dim u({\mathfrak b}^{+}/\operatorname{rad}(\beta_{\chi}),\chi)= 
(\dim S)^{2}.$$
Hence, the algebra $u({\mathfrak b}^{+}/\operatorname{rad}(\beta_ 
{\chi}),\chi)$ is semisimple and
$S$ is projective. In this case since there is only one simple module  
for the algebra
the dimension is equal to the dimension of the simple times the  
dimension of its projective cover
(cf.\ \cite[(6.17) Proposition]{CR}).

(d) Recall that $\operatorname{Ext}^{n}_{u({\mathfrak g},\chi)}(L,L) 
\cong
\operatorname{Ext}^{n}_{u({\mathfrak b}^{+},\chi)}(S,L)$ by Frobenius  
reciprocity.
We can now apply the Lyndon-Hochschild-Serre spectral sequence and  
use part (a):
$$E_{2}^{i,j}=\text{Ext}^{i}_{u({\mathfrak b}^{+}/\operatorname{rad} 
(\beta_{\chi}),\chi)}(S,
\text{H}^{j}(u(\operatorname{rad}(\beta_{\chi})),L))\Rightarrow \text 
{Ext}^{i+j}_{u({\mathfrak b}^{+},\chi)}(S,L).$$
 From part (b), the semisimplicity of $u({\mathfrak b}^{+}/ 
\operatorname{rad}(\beta_{\chi}),\chi)$ forces
the spectral sequence to collapse and yields the statement of (d).
\end{proof}

We end this section by completing the example ${\mathfrak g}=W(1,1)$  
when $p=5$. In this case $r=2$ or $3$.
In both cases the support variety of $L$ as a $u({\mathfrak g},\chi)$- 
module has dimension equal to 1
\cite[Theorem 4.5, Example 4.6]{FN}. Therefore, $u({\mathfrak g},\chi) 
$ is uniserial and
$\text{Ext}^{1}_{u({\mathfrak g},\chi)}(L,L)\cong k$. If $r=2$,
$$\text{Ext}^{1}_{u({\mathfrak b}^{+},\chi)}(S,S)\cong k\oplus k$$
by Theorem 4.3(A). So in this case the lower bound in Proposition 4.3 
(B) becomes an equality.
We conjecture that this will hold true for $r$ even and any $p$. On  
the other hand, if $r=3$ then
the support variety of $S$ over $u({\mathfrak b}^{+},\chi)$ is one- 
dimensional
(cf.\ \cite[Theorem 3.8]{Fa})and
$$\text{Ext}^{1}_{u({\mathfrak b}^{+},\chi)}(S,S)\cong k.$$
This shows that the inequality in Proposition 4.3(B) in the odd case  
can be strict.

\subsection{$r(\chi)=p-1$:} In the last case the $\text{Ext}^{1}$  
information can be
directly interpreted from \cite[Theorem 2.6]{FN}.

\begin{thm} Let ${\mathfrak g}=W(1,1)$ and $r(\chi)=p-1$.
\begin{itemize}
\item[(a)] If ${\mathfrak g}^{\chi}$ is a torus then $u({\mathfrak g}, 
\chi)$ is
semisimple. In this case $\operatorname{Ext}^{1}_{u({\mathfrak g}, 
\chi)}(L,L)=0$ for
all simple $u({\mathfrak g},\chi)$-modules $L$.
\item[(b)] If ${\mathfrak g}^{\chi}$ is $p$-nilpotent then $u 
({\mathfrak g},\chi)$
  has only one non-semisimple block. The non-semisimple
block has one simple module $L$ whose projective cover is uniserial  
with two
composition factors isomorphic to $L$. Thus, $\operatorname{Ext}^{1}_ 
{u({\mathfrak g},\chi)}(L,L)=k$.
\end{itemize}
\end{thm}

\let\section=\oldsection

\def\scr{\mathcal}\def\cprime{$'$} \def\germ{\mathfrak}

\end{document}